\documentclass{amsart}
\usepackage{amsmath,amssymb}
\usepackage[dvips]{graphics}
\usepackage[dvipdfmx]{graphicx}
\usepackage[all]{xy}

\newtheorem{Theorem}{Theorem}[section]
\newtheorem{Lemma}[Theorem]{Lemma}
\newtheorem{Proposition}[Theorem]{Proposition}

\theoremstyle{definition}

\theoremstyle{remark}
\newtheorem{Remark}[Theorem]{Remark}

\def\labelenumi{\rm ({\makebox[0.8em][c]{\theenumi}})}
\def\theenumi{\arabic{enumi}}
\def\labelenumii{\rm {\makebox[0.8em][c]{(\theenumii)}}}
\def\theenumii{\roman{enumii}}

\renewcommand{\thefigure}
{\arabic{section}.\arabic{figure}}
\makeatletter
\@addtoreset{figure}{section}
\def\@thmcountersep{-}
\makeatother

\numberwithin{equation}{section}

\begin{document}

\title[Capturing links in spatial complete graphs]{Capturing links in spatial complete graphs}

\author{Ryo Nikkuni}
\address{Department of Mathematics, School of Arts and Sciences, Tokyo Woman's Christian University, 2-6-1 Zempukuji, Suginami-ku, Tokyo 167-8585, Japan}
\email{nick@lab.twcu.ac.jp}
\thanks{The author was supported by JSPS KAKENHI Grant Number JP19K03500.}



\date{}



\begin{abstract}
We say that a set of pairs of disjoint cycles $\Lambda(G)$ of a graph $G$ is linked if for any spatial embedding $f$ of $G$ there exists an element $\lambda$ of $\Lambda(G)$ such that the $2$-component link $f(\lambda)$ is nonsplittable, and also say minimally linked if none of its proper subsets are linked. In this paper, (1) we show that the set of all pairs of disjoint cycles of $G$ is minimally linked if and only if $G$ is essentially same as a graph in the Petersen family, and (2) for any two integers $p,q\ge 3$, we exhibit a minimally linked set of Hamiltonian $(p,q)$-pairs of cycles of the complete graph $K_{p+q}$ with at most eighteen elements. 
\end{abstract}

\maketitle

\section{Introduction}\label{intro}

Throughout this paper we work in the piecewise linear category. An embedding $f$ of a finite graph $G$ into the $3$-sphere is called a {\it spatial embedding} of $G$, and $f(G)$ is called a {\it spatial graph} of $G$. We denote the set of all spatial embeddings of $G$ by ${\rm SE}(G)$. Two spatial graphs $f(G)$ and $g(G)$ are said to be {\it ambient isotopic} and denoted by $f(G)\cong g(G)$ if there exists an orientation-preserving self-homeomorphism $\Phi$ of the $3$-sphere such that $\Phi(f(G)) = g(G)$. We call a subgraph of $G$ homeomorphic to the circle a {\it cycle} of $G$, and a cycle containing exactly $p$ vertices a {\it $p$-cycle}. We denote the set of all pairs of disjoint cycles of $G$ by $\Gamma^{(2)}(G)$, and the subset of $\Gamma^{(2)}(G)$ consisting of all pairs of a $p$-cycle and a $q$-cycle by $\Gamma_{p,q}(G)$. We call an element of $\Gamma_{p,q}(G)$ a {\it $(p,q)$-pair of cycles} of $G$. For an element $\lambda$ of $\Gamma^{(2)}(G)$ and an element $f$ of ${\rm SE}(G)$, we call $f(\lambda)$ a {\it constituent $2$-component link} of $f(G)$, and if $\lambda$ is a $(p,q)$-pair of cycles, then we also say that $f(\lambda)$ is {\it of type $(p,q)$}. In particular, if $\lambda$ contains all vertices of $G$, then we call $\lambda$ a {\it Hamiltonian pair of cycles} and $f(\lambda)$ a $2$-component {\it Hamiltonian link} of $f(G)$. 

A graph $G$ is said to be {\it intrinsically linked} if for any element $f$ in ${\rm SE}(G)$ there exists an element $\lambda$ of $\Gamma^{(2)}(G)$ such that $f(\lambda)$ is a nonsplittable $2$-component link. Conway--Gordon and Sachs independently proved that $K_{6}$ is intrinsically linked \cite{CG83}, \cite{S84}, where $K_{n}$ is the {\it complete graph} on $n$ vertices, that is the loopless graph consisting of $n$ vertices, a pair of whose distinct vertices is connected by exactly one edge. Actually they showed that every spatial graph $f(K_{6})$ has a nonsplittable Hamiltonian link of type $(3,3)$. In the case of $n\ge 7$, since $K_{n}$ contains a subgraph $H$ isomorphic to $K_{6}$, every spatial graph $f(K_{n})$ also has a nonsplittable constituent $2$-component link of type $(3,3)$. On the other hand, for any integer $n\ge 6$, Vesnin--Litvintseva showed that every spatial graph $f(K_{n})$ has a nonsplittable $2$-component Hamiltonian link \cite{YL10}. Moreover, the following is also known. 

\begin{Theorem}\label{ham_link}{\rm (Morishita--Nikkuni \cite{MN19_b})} 
Let $p,q\ge 3$ be two integers. Then every spatial graph $f(K_{p+q})$ has a nonsplittable Hamiltonian link of type $(p,q)$. 
\end{Theorem}

As the number of vertices $n$ increases, the number of the nonsplittable $2$-component Hamiltonian links of $f(K_{n})$ also fairly increases and their behavior seems to be elusive, see \cite{FM09}, \cite{AMT13} for example. Our purpose in this paper is to ensure that we capture a nonsplittable Hamiltonian link of any possible type $(p,q)$ in a spatial complete graph on $n\ge 6$ vertices by the image of a smaller family of Hamiltonian $(p,q)$-pairs of cycles. Let $\Lambda(G)$ be a subset of $\Gamma^{(2)}(G)$. We say that $\Lambda(G)$ is {\it linked} if for any element $f$ in ${\rm SE}(G)$ there exists an element $\lambda$ of $\Lambda(G)$ such that the $2$-component link $f(\lambda)$ is nonsplittable. For example, Theorem \ref{ham_link} says that $\Gamma_{p,q}(K_{p+q})$ is linked for any integers $p,q\ge 3$. Note that $G$ is intrinsically linked if and only if there exists a linked subset $\Lambda(G)$ of $\Gamma^{(2)}(G)$. Moreover, we say that a linked set $\Lambda(G)$ is {\it minimally linked} if every proper subset of $\Lambda(G)$ is not linked \cite[\S 5]{ildt12}. By the definition of minimal linkedness, it is clear that every linked set of pairs of cycles includes a minimally linked subset. In addition, we also have the following, which properly represents the characteristics of the minimally linked set of pairs of cycles.

\begin{Proposition}\label{uni_link}
Let $\Lambda(G)$ be a subset of $\Gamma^{(2)}(G)$ that is linked. Then $\Lambda(G)$ is minimally linked if and only if for any element $\lambda$ in $\Lambda(G)$ there exists an element $f_{\lambda}$ of ${\rm SE}(G)$ such that $f_{\lambda}(\lambda)$ is a nonsplittable link and $f_{\lambda}(\lambda')$ is a split link for any element $\lambda'$ in $\Lambda(G)\setminus \{\lambda\}$. 
\end{Proposition}

\begin{proof}
First we show the `if' part. By the assumption, $\Lambda(G)$ is linked and $\Lambda(G)\setminus \{\lambda\}$ is not linked for any element $\lambda$ in $\Lambda(G)$. Thus $\Lambda(G)$ is minimally linked. Next we show the `only if' part. Assume that $\Lambda(G)$ is minimally linked. Then for any element $\lambda$ in $\Lambda(G)$, the set $\Lambda(G)\setminus \{\lambda\}$ is not linked. Thus there exists an element $f_{\lambda}$ of ${\rm SE}(G)$ such that $f_{\lambda}(\lambda')$ is a split link for any element $\lambda'$ in $\Lambda(G)\setminus \{\lambda\}$. Since $\Lambda(G)$ is linked, the link $f_{\lambda}(\lambda)$ must be nonsplittable. 
\end{proof}

Our first result in this paper is to reveal the relationship between the minimality of linked set of pairs of cycles and the minor-minimality of the intrinsic linkedness. An {\it edge contraction} on a graph is an operation that contracts an edge $e$ of the graph which is not a loop to a new vertex $v$ as illustrated in Fig. \ref{edgecont}. The reverse operation of an edge contraction is called a {\it vertex splitting}. We say that an edge contraction (resp. vertex splitting) is {\it topologically trivial} if it does not change the topological type of the graph, or equivalently, the degree of either terminal vertices of $e$ is $2$. A graph $H$ is called a {\it minor} of a graph $G$ if there exists a subgraph $G'$ of $G$ such that $H$ is obtained from $G'$ by finite number of edge contractions. In particular, $H$ is called a {\it proper minor} of $G$ if $H\neq G$. 
\begin{figure}[htbp]
\begin{center}
\scalebox{0.5}{\includegraphics*{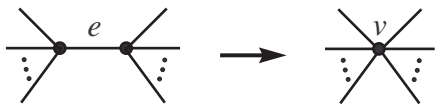}}
\caption{Edge contraction}
\label{edgecont}
\end{center}
\end{figure}
An intrinsically linked graph $G$ is said to be {\it minor-minimal} if every proper minor of $G$ is not intrinsically linked. It is well-known that every intrinsically linked graph has a minor-minimal intrinsically linked graph as a minor, and all minor-minimal intrinsically linked graphs are only graphs in the {\it Petersen family} that is a family of exactly seven graphs as illustrated in Fig. \ref{capture_link2} \cite{RST95}. Then the following says that for an intrinsically linked graph $G$, the minimality of $\Gamma^{(2)}(G)$ and the minor-minimality of $G$ are equivalent.

\begin{Theorem}\label{lk_min_petersen}
Let $G$ be a graph with no vertices of degree $0$ or $1$. Then $\Gamma^{(2)}(G)$ is minimally linked if and only if $G$ is homeomorphic to one of the graphs in the Petersen family. 
\end{Theorem}

\begin{figure}[htbp]
\begin{center}
\scalebox{0.625}{\includegraphics*{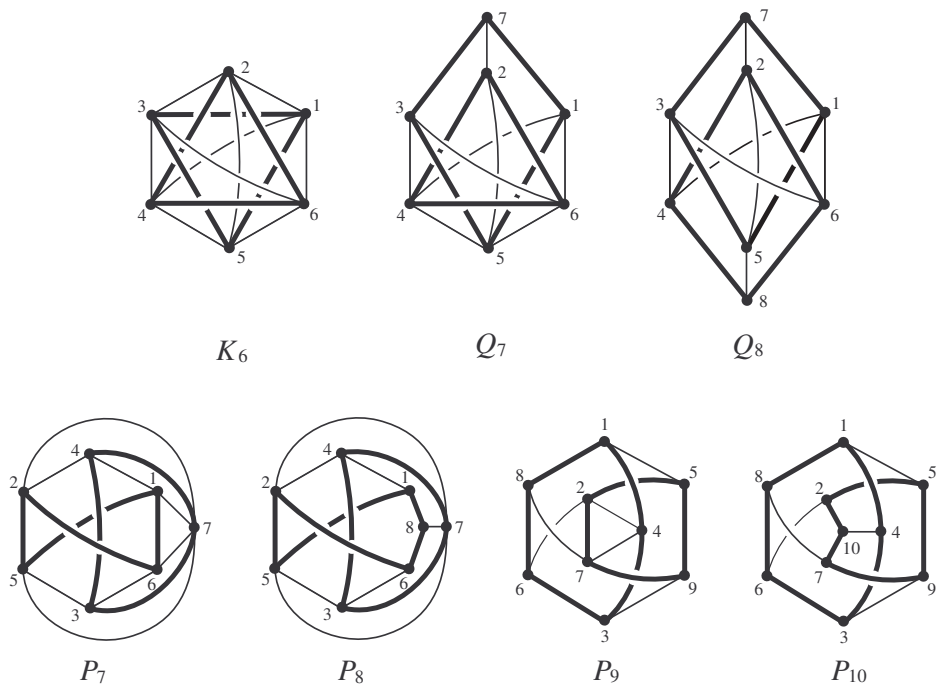}}
\caption{Petersen family $K_{6}$, $Q_{7}$, $Q_{8}$, $P_{7}$, $P_{8}$, $P_{9}$, $P_{10}$}
\label{capture_link2}
\end{center}
\end{figure}

Our second result in this paper is to explicitly give a minimally linked subset of $\Gamma_{p,q}(K_{p+q})$ for any $p,q\ge 3$. To accomplish this, we give a mimimally linked set of Hamiltonian pairs of cycles for some specific graphs as follows. Let $G$ be one of the graphs $G_{8},G_{9}$ and $G_{10}$ as illustrated in Fig. \ref{capture_link}. We denote an edge of $G$ connecting two vertices $i$ and $j$ by $\overline{ij}$, and a $p$-cycle $\overline{i_{1}i_{2}}\cup \overline{i_{2}i_{3}}\cup \cdots \cup \overline{i_{p}i_{1}}$ of $G$ by $[i_{1}\ i_{2}\ \cdots\ i_{p}]$. Then we define the subset $\Lambda(G)$ of $\Gamma^{(2)}(G)$ as follows. For $G = G_{8}$, we define $\Lambda(G_{8})$ by the proper subset of $\Gamma_{5,3}(G_{8})$ consisting of the following twelve Hamiltonian $(5,3)$-pairs of cycles 
\begin{eqnarray*}
&&[1\ 8\ 7\ 2\ 3]\cup [4\ 5\ 6],\ [1\ 2\ 8\ 7\ 3]\cup [4\ 5\ 6],\ [1\ 2\ 3\ 8\ 7]\cup [4\ 5\ 6], \\
&&[1\ 8\ 7\ 2\ 4]\cup [3\ 5\ 6],\ [1\ 8\ 7\ 2\ 5]\cup [4\ 3\ 6],\ [1\ 8\ 7\ 2\ 6]\cup [4\ 5\ 3],\\
&&[6\ 2\ 8\ 7\ 3]\cup [4\ 5\ 1],\ [5\ 2\ 8\ 7\ 3]\cup [4\ 1\ 6],\ [4\ 2\ 8\ 7\ 3]\cup [1\ 5\ 6],\\
&&[1\ 4\ 3\ 8\ 7]\cup [2\ 5\ 6],\ [1\ 5\ 3\ 8\ 7]\cup [4\ 2\ 6],\ [1\ 6\ 3\ 8\ 7]\cup [4\ 5\ 2]. 
\end{eqnarray*}
For $G = G_{9}$, we define $\Lambda(G_{9})$ by the proper subset of $\Gamma_{5,4}(G_{9})$ consisting of the following nine Hamiltonian $(5,4)$-pairs of cycles 
\begin{eqnarray*}
&&[1\ 8\ 7\ 2\ 6]\cup [4\ 9\ 5\ 3],\ [1\ 8\ 7\ 2\ 4]\cup [3\ 5\ 9\ 6],\ [1\ 8\ 7\ 2\ 5]\cup [4\ 3\ 6\ 9], \\
&&[6\ 2\ 8\ 7\ 3]\cup [4\ 9\ 5\ 1],\ [4\ 2\ 8\ 7\ 3]\cup [1\ 5\ 9\ 6],\ [5\ 2\ 8\ 7\ 3]\cup [4\ 1\ 6\ 9],\\
&&[1\ 6\ 3\ 8\ 7]\cup [4\ 9\ 5\ 2],\ [1\ 4\ 3\ 8\ 7]\cup [2\ 5\ 9\ 6],\ [1\ 5\ 3\ 8\ 7]\cup [4\ 2\ 6\ 9]. 
\end{eqnarray*}
For $G = G_{10}$, we define $\Lambda(G_{10})$ by the proper subset of $\Gamma_{5,5}(G_{10})$ consisting of the following eighteen Hamiltonian $(5,5)$-pairs of cycles 
\begin{eqnarray*}
&&[1\ 8\ 7\ 2\ 3]\cup [4\ 10\ 9\ 5\ 6],\ [1\ 8\ 7\ 2\ 3]\cup [4\ 5\ 10\ 9\ 6],\ [1\ 8\ 7\ 2\ 3]\cup [4\ 5\ 6\ 10\ 9], \\
&&[1\ 2\ 8\ 7\ 3]\cup [4\ 10\ 9\ 5\ 6],\ [1\ 2\ 8\ 7\ 3]\cup [4\ 5\ 10\ 9\ 6],\ [1\ 2\ 8\ 7\ 3]\cup [4\ 5\ 6\ 10\ 9], \\
&&[1\ 2\ 3\ 8\ 7]\cup [4\ 10\ 9\ 5\ 6],\ [1\ 2\ 3\ 8\ 7]\cup [4\ 5\ 10\ 9\ 6],\ [1\ 2\ 3\ 8\ 7]\cup [4\ 5\ 6\ 10\ 9], \\
&&[1\ 8\ 7\ 2\ 6]\cup [4\ 10\ 9\ 5\ 3],\ [1\ 8\ 7\ 2\ 4]\cup [3\ 5\ 10\ 9\ 6],\ [1\ 8\ 7\ 2\ 5]\cup [4\ 3\ 6\ 10\ 9],\\
&&[6\ 2\ 8\ 7\ 3]\cup [4\ 10\ 9\ 5\ 1],\ [4\ 2\ 8\ 7\ 3]\cup [1\ 5\ 10\ 9\ 6],\ [5\ 2\ 8\ 7\ 3]\cup [4\ 1\ 6\ 10\ 9], \\
&&[1\ 6\ 3\ 8\ 7]\cup [4\ 10\ 9\ 5\ 2],\ [1\ 4\ 3\ 8\ 7]\cup [2\ 5\ 10\ 9\ 6],\ [1\ 5\ 3\ 8\ 7]\cup [4\ 2\ 6\ 10\ 9]. 
\end{eqnarray*}

\begin{Theorem}\label{maing8910}
Let $G$ be one of the graphs $G_{8},G_{9}$ and $G_{10}$. Then $\Lambda(G)$ is minimally linked. 
\end{Theorem}

\begin{figure}[htbp]
\begin{center}
\scalebox{0.525}{\includegraphics*{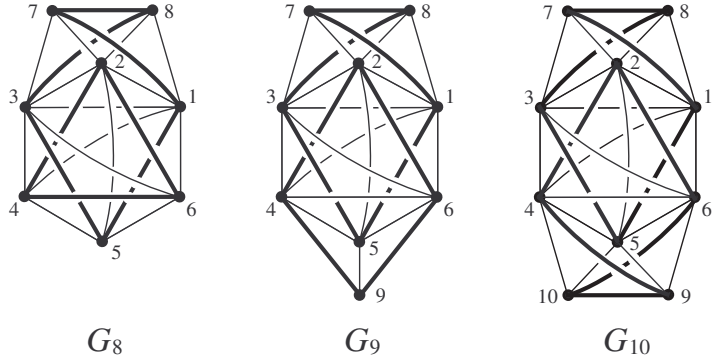}}
\caption{$G_{8}$, $G_{9}$, $G_{10}$}
\label{capture_link}
\end{center}
\end{figure}

By using Theorem \ref{maing8910}, we have the following for the complete graph $K_{n} $ on $n\ge 7$ vertices.

\begin{Theorem}\label{main}
\begin{enumerate}
\item For any integer $p\ge 5$, there exists a minimally linked subset of $\Gamma_{p,3}(K_{p+3})$ with exactly twelve elements. 
\item For any integer $p\ge 3$, there exists a minimally linked subset of $\Gamma_{p,4}(K_{p+4})$ with exactly nine elements. 
\item For any two integers $p,q\ge 5$, there exists a minimally linked subset of $\Gamma_{p,q}(K_{p+q})$ with exactly eighteen elements. 
\end{enumerate}
\end{Theorem}

The proof is constructive, namely a minimally linked set of Hamiltonian pairs of cycles is explicitly given in any case. Therefore, for any two integers $p,q\ge 3$, by lying in ambush on at most eighteen specific Hamiltonian $(p,q)$-pairs of cycles, we can capture a nonsplittable Hamiltonian link of type $(p,q)$ for any spatial graph of $K_{p+q}$. We prove Theorem \ref{lk_min_petersen} in Section \ref{proofs1}, and Theorem \ref{maing8910} and Theorem \ref{main} in Section \ref{proofs2}.

\section{Proof of Theorem \ref{lk_min_petersen}}\label{proofs1}

We have already known that the set of all pairs of disjoint cycles of a graph in the Petersen family is linked \cite{S84}. First, we show their minimal linkedness. 

\begin{Lemma}\label{min_petersen}
Let $G$ be a graph in the Petersen family. Then $\Gamma^{(2)}(G)$ is minimally linked. 
\end{Lemma}

In the following, an {\it automorphism} of $G$ means a self-homeomorphism of $G$, and we may identify an automorphism of $G$ with a permutation of degree $n$, where $n$ is the number of vertices of $G$. 

\begin{proof}[Proof of Lemma \ref{min_petersen}]
Let $h$ be a spatial embedding of $G$ as illustrated in Fig. \ref{capture_link2}. Then we can check that each of the spatial graphs contains a Hopf link in the thick line part as exactly one nonsplittable link in the images of all elements in $\Gamma^{(2)}(G)$. 

In the case of $K_{6}$, for any element $\lambda$ in $\Gamma^{(2)}(K_{6})=\Gamma_{3,3}(K_{6})$ there exists an automorphism $\sigma$ of $K_{6}$ such that $\sigma$ sends  $\lambda$ to $[1\ 3\ 5]\cup [2\ 4\ 6]$. Then $h\circ \sigma(\lambda)$ is a Hopf link and $h\circ \sigma(\lambda')$ is a split link for any element $\lambda'$ in $\Gamma^{(2)}(K_{6})\setminus \{\lambda\}$. Thus by Proposition \ref{uni_link}, $\Gamma^{(2)}(K_{6})$ is minimally linked. 

In the case of $G=Q_{7},Q_{8},P_{7}$, for any element $\lambda$ in $\Gamma^{(2)}(G)=\Gamma_{4,3}(G)$ if $G = G_{7}$ or $P_{7}$, and $\Gamma_{4,4}(G)$ if $G=Q_{8}$, there exists an automorphism $\sigma$ of $G$ generated by $(1\ 2\ 3)$ and $(4\ 5\ 6)$ such that $\sigma$ sends $\lambda$ to $[1\ 7\ 3\ 5]\cup [2\ 4\ 6]$ if $G=Q_{7}$, $[1\ 5\ 2\ 6]\cup [7\ 3\ 4]$ if $G=P_{7}$ and $[1\ 7\ 3\ 5]\cup [2\ 4\ 8\ 6]$ if $G=Q_{8}$. Then $h\circ \sigma(\lambda)$ is a Hopf link and $h\circ \sigma(\lambda')$ is a split link for any element $\lambda'$ in $\Gamma^{(2)}(G)\setminus \{\lambda\}$. Thus by Proposition \ref{uni_link}, $\Gamma^{(2)}(G)$ is minimally linked. 

In the case of $P_{8}$, note that $\Gamma^{(2)}(P_{8})=\Gamma_{5,3}(P_{8})\cup \Gamma_{4,4}(P_{8})$. For any element $\lambda$ in $\Gamma_{5,3}(P_{8})$, there exists an automorphism $\sigma$ of $P_{8}$ generated by $(2\ 3)$ and $(4\ 5)$ such that $\sigma$ sends $\lambda$ to $[1\ 5\ 2\ 6\ 8]\cup [7\ 3\ 4]$. Then $h\circ \sigma(\lambda)$ is a Hopf link and $h\circ \sigma(\lambda')$ is a split link for any element $\lambda'$ in $\Gamma^{(2)}(P_{8})\setminus \{\lambda\}$. On the other hand, for any element $\mu$ in $\Gamma_{4,4}(P_{8})$, there exists an automorphism $\tau$ of $P_{8}$ generated by $(2\ 3)$, $(4\ 5)$ and $(1\ 6)(3\ 4)(2\ 5)$ such that $\tau$ sends $\mu$ to $[8\ 1\ 5\ 7]\cup [4\ 2\ 6\ 3]$. Let $g$ be an element of ${\rm SE}(P_{8})$ obtained from $h$ by a single crossing change between $h(\overline{15})$ and $h(\overline{34})$ and ambient isotopies such that $g([1\ 5\ 2\ 6\ 8]\cup [7\ 3\ 4])$ is a trivial link. Then we have that $g([8\ 1\ 5\ 7]\cup [4\ 2\ 6\ 3])$ has nonzero linking number, namely this is nonsplittable. Since no other element in $\Gamma^{(2)}(P_{8})$ contains $\overline{15}$ and $\overline{34}$ in the different components, $g(\lambda')$ is a split link for any element $\lambda'$ in $\Gamma^{(2)}(P_{8})\setminus \{[8\ 1\ 5\ 7]\cup [4\ 2\ 6\ 3]\}$. Then $g\circ \tau(\mu)$ is a nonsplittable link and $g\circ \tau(\mu')$ is a split link for any element $\mu'$ in $\Gamma^{(2)}(P_{8})\setminus \{\mu\}$. Thus by Proposition \ref{uni_link}, $\Gamma^{(2)}(P_{8})$ is minimally linked.

In the case of $P_{9}$, note that $\Gamma^{(2)}(P_{9})=\Gamma_{5,4}(P_{9})\cup \{\mu\}$, where $\mu = [1\ 8\ 6\ 3\ 9\ 5]\cup [2\ 7\ 4]$. For any element $\lambda$ in $\Gamma_{5,4}(P_{9})$, there exists an automorphism $\sigma$ of $P_{9}$ generated by $(1\ 6\ 9)(8\ 3\ 5)(2\ 7\ 4)$ and $(4\ 7\ 2)(1\ 8\ 6\ 3\ 9\ 5)$ such that $\sigma$ sends $\lambda$ to $[1\ 8\ 6\ 3\ 4]\cup [5\ 2\ 7\ 9]$. Then $h\circ \sigma(\lambda)$ is a Hopf link and $h\circ \sigma(\lambda')$ is a split link for any element $\lambda'$ in $\Gamma^{(2)}(P_{9})\setminus \{\lambda\}$. On the other hand, let $g$ be an element of ${\rm SE}(P_{9})$ obtained from $h$ by a single crossing change between $h(\overline{27})$ and $h(\overline{86})$ and ambient isotopies such that $g([1\ 8\ 6\ 3\ 4]\cup [5\ 2\ 7\ 9])$ is a trivial link. Then we have that $g(\mu)$ has nonzero linking number, namely this is nonsplittable. Since no other element in $\Gamma^{(2)}(P_{9})$ contains $\overline{27}$ and $\overline{86}$ in the different components, $g(\mu')$ is a split link for any element $\mu'$ in $\Gamma^{(2)}(P_{8})\setminus \{\mu\}$. Thus by Proposition \ref{uni_link}, $\Gamma^{(2)}(P_{9})$ is minimally linked.

In the case of $P_{10}$, for any element $\lambda$ in $\Gamma^{(2)}(P_{10})=\Gamma_{5,5}(P_{10})$ there exists an automorphism $\sigma$ of $P_{10}$ generated by $(1\ 6\ 9)(8\ 3\ 5)(2\ 7\ 4)$ and $(4\ 7\ 2)(1\ 8\ 6\ 3\ 9\ 5)$ such that $\sigma$ sends  $\lambda$ to $[1\ 8\ 6\ 3\ 4]\cup [5\ 2\ 10\ 7\ 9]$. Then $h\circ \sigma(\lambda)$ is a Hopf link and $h\circ \sigma(\lambda')$ is a split link for any element $\lambda'$ in $\Gamma^{(2)}(P_{10})\setminus \{\lambda\}$. Thus by Proposition \ref{uni_link}, $\Gamma^{(2)}(P_{10})$ is minimally linked. 
\end{proof}

Let $H$ be a minor of a graph $G$, namely there exists a subgraph $G'$ of $G$ and edges $e_{1},e_{2},\ldots,e_{m}$ of $G'$ such that $H$ is obtained from $G'$ by edge contractions along $e_{1},e_{2},\ldots,e_{m}$. Then by composing the injective map from $\Gamma^{(2)}(H)$ to $\Gamma^{(2)}(G')$ induced from vertex splittings on $H$ and the inclusion from $\Gamma^{(2)}(G')$ to $\Gamma^{(2)}(G)$, we obtain the natural injective map  
\begin{eqnarray*}
\Psi_{H,G}^{(2)}:\Gamma^{(2)}(H)\longrightarrow \Gamma^{(2)}(G). 
\end{eqnarray*}
On the other hand, for an element $f$ of ${\rm SE}(G)$, the element $\psi(f)$ of ${\rm SE}(H)$ is obtained from $f(G')$ by contracting spatial edges $f(e_{i})\ (i=1,2,\ldots,m)$. This correspondence from $f$ to $\psi(f)$ defines a surjective map 
\begin{eqnarray*}
\psi:{\rm SE}(G)\longrightarrow {\rm SE}(H). 
\end{eqnarray*}
Note that we can easily see the following.

\begin{Proposition}\label{psipsi}
Let $f$ be an element of ${\rm SE}(G)$. Then for any element $\lambda$ in $\Gamma^{(2)}(H)$, $2$-component links $\psi(f)(\lambda)$ and $f\big(\Psi_{H,G}^{(2)}(\lambda)\big)$ are ambient isotopic.
\end{Proposition}

Then the next lemma says that a minimally linked set of pairs of cycles for a minor $H$ of a graph $G$ naturally induces a minimally linked set for $G$.

\begin{Lemma}\label{minor_minlk}
Let $\Lambda(H)$ be a subset of $\Gamma^{(2)}(H)$ and $\Lambda(G)$ the image of $\Lambda(H)$ by $\Psi_{H,G}^{(2)}$. If $\Lambda(H)$ is (minimally) linked then $\Lambda(G)$ is also (minimally) linked. 
\end{Lemma}

\begin{proof}
Let $f$ be an element of ${\rm SE}(G)$. Then by the assumption that $\Lambda(H)$ is linked and Proposition \ref{psipsi}, there exists an element $\lambda_{0}$ of $\Lambda(H)$ such that $f(\Psi_{H,G}^{(2)}(\lambda_{0}))\cong \psi(f)(\lambda_{0})$ is a nonsplittable link. Thus $\Lambda(G)=\Psi_{H,G}^{(2)}(\Lambda(H))$ is linked. Suppose that $\Lambda(H)$ is minimally linked. Then for any element $\Psi_{H,G}^{(2)}(\lambda)$ in $\Lambda(G)$, there exists an element $\bar{f}_{\lambda}$ of ${\rm SE}(H)$ such that $\bar{f}_{\lambda}(\lambda)$ is a nonsplittable link and $\bar{f}_{\lambda}(\lambda')$ is a split link for any element $\lambda'$ in $\Lambda(H)\setminus \{\lambda\}$. Then for an element $f_{\lambda}$ of $\psi^{-1}(\bar{f}_{\lambda})$, by Proposition \ref{psipsi} we have that $f_{\lambda}(\Psi_{H,G}^{(2)}(\lambda))\cong \bar{f}_{\lambda}(\lambda)$ is a nonsplittable link and $f_{\lambda}(\Psi_{H,G}^{(2)}(\lambda'))\cong \bar{f}_{\lambda}(\lambda')$ is a split link for $\lambda'\neq \lambda$. Since $\Psi_{H,G}^{(2)}$ is injective, we have $\Psi_{H,G}^{(2)}(\lambda')\neq \Psi_{H,G}^{(2)}(\lambda)$. Thus $\Lambda(G) = \Psi_{H,G}^{(2)}(\Lambda(H))$ is minimally linked. 
\end{proof}

A {\it cut-edge} of a graph is an edge of the graph whose deletion increases the number of the connected components. Then the next lemma says that if $\Gamma^{(2)}(G)$ is minimally linked then $G$ is essentially `$2$-edge-connected'. 

\begin{Lemma}\label{2conn}
Let $G$ be a graph with no vertices of degree $0$ or $1$. If $\Gamma^{(2)}(G)$ is minimally linked then $G$ is connected and does not have a cut-edge. 
\end{Lemma}

\begin{proof}
First, assume that there exists a disconnected graph $G$ such that $\Gamma^{(2)}(G)$ is minimally linked. Let $G_{1},G_{2},\ldots,G_{m}$ be the connected components of $G$ ($m\ge 2$). Then for a pair of cycles $\lambda = \gamma_{1}\cup \gamma_{2}$ where $\gamma_{i}$ is a cycle of $G_{i}\ (i=1,2)$, by Proposition \ref{uni_link}, there exists an element $f_{\lambda}$ of ${\rm SE}(G)$ such that $f_{\lambda}(\lambda)$ is a nonsplittable link and $f_{\lambda}(\lambda')$ is a split link for any element $\lambda'$ in $\Gamma^{(2)}(G)\setminus \{\lambda\}$. This implies that $\Gamma^{(2)}(G_{i})$ is not linked for any $i=1,2,\ldots,m$. Then we can see that there exists an element $g$ of ${\rm SE}(G)$ such that $g(G)$ does not contain a nonsplittable $2$-component link. This is a contradiction. Hence, if $\Gamma^{(2)}(G)$ is minimally linked then $G$ is connected. 

Next, assume that $G$ has a cut-edge $e$. Let $G_{1}$ and $G_{2}$ be the connected components of $G\setminus {\rm int}e$. Then for a pair of cycles $\lambda = \gamma_{1}\cup \gamma_{2}$ where $\gamma_{i}$ is a cycle of $G_{i}\ (i=1,2)$, by Proposition \ref{uni_link}, there exists an element $f_{\lambda}$ of ${\rm SE}(G)$ such that $f_{\lambda}(\lambda)$ is a nonsplittable link and $f_{\lambda}(\lambda')$ is a split link for any element $\lambda'$ in $\Gamma^{(2)}(G)\setminus \{\lambda\}$. This implies that $\Gamma^{(2)}(G_{i})$ is not linked for any $i=1,2$. Then we also can see that there exists an element $g$ of ${\rm SE}(G)$ such that $g(G)$ does not contain a nonsplittable $2$-component link. This is a contradiction. 
\end{proof}

Now let us prove Theorem \ref{lk_min_petersen}. In the following, a {\it subdivision} $H'$ of a graph $H$ is a graph obtained from $H$ by subdividing some edges by a finite number of vertices, or equivalently, $H'$ is obtained from $H$ by finite number of vertex splittings which are topologically trivial. Thus a graph $H$ is a minor of its subdivision $H'$. 

\begin{proof}[Proof of Theorem \ref{lk_min_petersen}]
First we show the `if' part. Let $G$ be a graph homeomorphic to a graph $P$ in the Petersen family. Then $G$ is a subdivision of $P$ and the map $\Psi_{P,G}^{(2)}:\Gamma^{(2)}(P)\to \Gamma^{(2)}(G)$ is bijective. Thus by Lemma \ref{min_petersen} and Lemma \ref{minor_minlk}, we have that $\Gamma^{(2)}(G) = \Psi_{P,G}^{(2)}(\Gamma^{(2)}(P))$ is minimally linked. 

Next we show the `only if' part. Suppose that $\Gamma^{(2)}(G)$ is minimally linked. Since $G$ is intrinsically linked, $G$ has a minor $P$ in the Petersen family, namely there exists a subgraph $G'$ of $G$ such that $P$ is obtained from $G'$ by some edge contractions. Here, $G'$ can be taken as having no vertices of degree $0$ or $1$. Let $\Psi_{P,G}^{(2)}:\Gamma^{(2)}(P)\to \Gamma^{(2)}(G)$ be the natural injective map. Then by Lemma \ref{min_petersen} and Lemma \ref{minor_minlk}, we have that $\Psi_{P,G}^{(2)}(\Gamma^{(2)}(P))$ is linked. Since $\Gamma^{(2)}(G)$ is minimally linked, $\Psi_{P,G}^{(2)}(\Gamma^{(2)}(P))$ must coincide with $\Gamma^{(2)}(G)$. This implies that the map $\Psi_{P,G}^{(2)}$ is bijective. Let $P'$ be a graph obtained from $P$ by a single vertex splitting at a vertex $v$. If the degree of $v$ is greater than or equal to $4$ and this vertex splitting is not topologically trivial, then it can be seen that $\sharp \Gamma^{(2)}(P') > \sharp \Gamma^{(2)}(P)$ by checking each of the graphs in the Petersen family one by one. This implies that $G'$ is a subdivision of $P$. Suppose that $G\neq G'$. Then we divide our situation into the following two cases. First we consider the case that $G'$ is a spanning subgraph of $G$. Let $e$ be an edge of $G$ not contained in $G'$. Note that for any pair of vertices $u,v$ of $P$ (possibly $u=v$), there exist a shortest path $\rho$ between $u$ and $v$ and a cycle $\gamma$ of $P$ such that $\rho$ and $\gamma$ are disjoint. Therefore, if both of the terminal vertices of $e$ are the original vertices $u,v$ of $P$ in the subdivision $G'$, then there exist a path $\rho'$ between $u$ and $v$ and a cycle $\gamma'$ of $G'$ such that $e\cup \rho'$ and $\gamma'$ are disjoint pair of cycles in $G$ not contained in $G'$. Thus we have $\sharp \Gamma^{(2)}(P) = \sharp \Gamma^{(2)}(G') < \sharp \Gamma^{(2)}(G)$, and this is a contradiction. If either of the terminal vertices of $e$ is not a vertex of $P$ in $G'$, then we can take another subdivision $G''$ of $P$ so that both of the terminal vertices of $e$ are the vertices of $P$ in $G''$. Thus, similar to the previous argument, we find that there is a contradiction. Next we consider the case that $G'$ is not a spanning subgraph of $G$. Let $F'$ be a connected component of a subgraph of $G$ induced by the vertices of $G$ which are not contained in $G'$. Note that by Lemma \ref{2conn}, $G$ does not have a cut-edge. Thus there exist at least two edges connecting $G'$ and $F'$. Hence, there exists a path of $G$ connecting the vertices of $G'$ which is edge-disjoint with $G'$. Then by the same argument as in the previous case, we have $\sharp \Gamma^{(2)}(P) = \sharp \Gamma^{(2)}(G') < \sharp \Gamma^{(2)}(G)$, and this is a contradiction. From the above we have $G = G'$, namely $G$ is homeomorphic to $P$. 
\end{proof}

\section{Proof of Theorem \ref{maing8910} and Theorem \ref{main}}\label{proofs2}

For a graph $G=G_{8},G_{9},G_{10}$, we first show the linkedness of $\Lambda(G)$. 

\begin{Lemma}\label{cg_gen}
Let $G$ be one of the graphs $G_{8},G_{9}$ and $G_{10}$. Then $\Lambda(G)$ is linked. 
\end{Lemma} 

\begin{proof}
The proof is given in exactly the same way as Conway--Gordon theorem for $K_{6}$ \cite{CG83}. Let $f$ be a spatial embedding of $G$. In the following we show 
\begin{eqnarray*}
\sum_{\lambda\in \Lambda(G)}{\rm lk}(f(\lambda))\equiv 1\pmod{2}, 
\end{eqnarray*}
where ${\rm lk}$ denotes the linking number in the $3$-sphere. This implies that there exists an element $\lambda$ in $\Lambda(G)$ such that $f(\lambda)$ has odd linking number, namely $f(\lambda)$ is nonsplittable. 
%
%
We define $\varsigma(f)\in {\mathbb Z}_{2}$ by the modulo two reduction of $\sum_{\lambda\in \Lambda(G)}{\rm lk}(f(\lambda))$. Note that the mod $2$ linking number of a $2$-component link does not change under crossing changes on the same component, and changes under a single crossing change between different components. So $\varsigma(f)$ does not change under crossing changes on the same edge, and between adjacent edges. On the other hand, it can be checked that for any disjoint edges $e$ and $e'$ of $G$, there exist even number of pairs of cycles in $\Lambda(G)$ containing both $e$ and $e'$ in each of the components separately. So $\varsigma(f)$ does not change under crossing changes between disjoint edges. Therefore $\varsigma(f)$ does not change under any crossing change on $f(G)$. Since any two spatial embeddings of a graph are transformed into each other by crossing changes and ambient isotopies, this implies that the value $\varsigma(f)$ does not depend on the choice of $f$. Let $h$ be a spatial embedding of $G$ as illustrated in Fig. \ref{capture_link}. Then we can check that each of the spatial graphs contains a Hopf link in the thick line part as exactly one nonsplittable link in the images of all elements in $\Lambda(G)$. Thus we have $\varsigma(f) = \varsigma(h)=1$.
\end{proof}

\begin{proof}[Proof of Theorem \ref{maing8910}]
In the case of $G_{8}$, for any element $\lambda$ in $\Lambda(G_{8})$ which does not contain $[4\ 5\ 6]$ as one of the components, there exists an automorphism $\sigma$ of $G_{8}$ generated by $(1\ 2\ 3)$ and $(4\ 5\ 6)$ such that $\sigma$ sends $\lambda$ to $[1\ 5\ 3\ 8\ 7]\cup [4\ 2\ 6]$. Let $h$ be an element of ${\rm SE}(G_{8})$ as illustrated in Fig. \ref{capture_link}. Then $h\circ \sigma(\lambda)$ is a Hopf link and $h\circ \sigma(\lambda')$ is a split link for any element $\lambda'$ in $\Lambda(G_{8})\setminus \{\lambda\}$. On the other hand, for any element $\mu$ in $\Lambda(G_{8})$ containing $[4\ 5\ 6]$ as one of the components, there exists an automorphism $\tau$ of $G_{8}$ generated by $(1\ 2\ 3)$ such that $\tau$ sends $\mu$ to $[1\ 2\ 3\ 8\ 7]\cup [4\ 5\ 6]$. Let $g$ be an element of ${\rm SE}(G_{8})$ obtained from $h$ by a single crossing change between $h(\overline{38})$ and $h(\overline{46})$ and ambient isotopies such that $g([1\ 5\ 3\ 8\ 7]\cup [4\ 2\ 6])$ is a trivial link. Then we have that $g([1\ 2\ 3\ 8\ 7]\cup [4\ 5\ 6])$ has nonzero linking number, namely this is nonsplittable. Since no other element in $\Lambda(G_{8})$ contains both $\overline{38}$ and $\overline{46}$, $g(\lambda')$ is a split link for any element $\lambda'$ in $\Lambda(G_{8})\setminus \{[1\ 2\ 3\ 8\ 7]\cup [4\ 5\ 6]\}$. Then $g\circ \tau(\mu)$ is a nonsplittable link and $g\circ \tau(\mu')$ is a split link for any element $\mu'$ in $\Lambda(G_{8})\setminus \{\mu\}$. Thus by Proposition \ref{uni_link}, $\Lambda(G_{8})$ is minimally linked. 

In the case of $G_{9}$, for any element $\lambda$ in $\Lambda(G_{9})$ there exists an automorphism $\sigma$ of $G_{9}$ generated by $(1\ 2\ 3)$ and $(4\ 5\ 6)$ such that $\sigma$ sends $\lambda$ to $[1\ 5\ 3\ 8\ 7]\cup [4\ 2\ 6\ 9]$. Let $h$ be an element of ${\rm SE}(G_{9})$ as illustrated in Fig. \ref{capture_link}. Then $h\circ \sigma(\lambda)$ is a Hopf link and $h\circ \sigma(\lambda')$ is a split link for any element $\lambda'$ in $\Lambda(G_{9})\setminus \{\lambda\}$. Thus by Proposition \ref{uni_link}, $\Lambda(G_{9})$ is minimally linked. 

In the case of $G_{10}$, for any element $\lambda$ in $\Lambda(G_{10})$ where not all of the vertices $1, 2$ and $3$ (equivalently, $4, 5$ and $6$) are included in the same component, there exists an automorphism $\sigma$ of $G_{8}$ generated by $(1\ 2\ 3)$ and $(4\ 5\ 6)$ such that $\sigma$ sends $\lambda$ to $[1\ 5\ 3\ 8\ 7]\cup [4\ 2\ 6\ 10\ 9]$. Let $h$ be an element of ${\rm SE}(G_{10})$ as illustrated in Fig. \ref{capture_link}. Then $h\circ \sigma(\lambda)$ is a Hopf link and $h\circ \sigma(\lambda')$ is a split link for any element $\lambda'$ in $\Lambda(G_{10})\setminus \{\lambda\}$. On the other hand, for any element $\mu$ in $\Lambda(G_{10})$ where all of the vertices $1, 2$ and $3$ (equivalently, $4, 5$ and $6$) are included in the same component, there exists an automorphism $\tau$ of $G_{10}$ generated by $(1\ 2\ 3)$ and $(4\ 5\ 6)$ such that $\tau$ sends $\mu$ to $[1\ 2\ 3\ 8\ 7]\cup [4\ 5\ 6\ 10\ 9]$. Let $g$ be an element of ${\rm SE}(G_{10})$ obtained from $h$ by a single crossing change between $h(\overline{6\ 10})$ and $h(\overline{38})$ and ambient isotopies such that $g([1\ 5\ 3\ 8\ 7]\cup [4\ 2\ 6\ 10\ 9])$ is a trivial link. Then we have that $g([1\ 2\ 3\ 8\ 7]\cup [4\ 5\ 6\ 10\ 9])$ has nonzero linking number, namely this is nonsplittable. Since no other element in $\Lambda(G_{10})$ contains both $\overline{6\ 10}$ and $\overline{38}$, $g(\lambda')$ is a split link for any element $\lambda'$ in $\Lambda(G_{10})\setminus \{[1\ 2\ 3\ 8\ 7]\cup [4\ 5\ 6\ 10\ 9]\}$. Then $g\circ \tau(\mu)$ is a nonsplittable link and $g\circ \tau(\mu')$ is a split link for any element $\mu'$ in $\Lambda(G_{10})\setminus \{\mu\}$. Thus by Proposition \ref{uni_link}, $\Lambda(G_{10})$ is minimally linked. 
\end{proof}

\begin{proof}[Proof of Theorem \ref{main}]
(1) Let $p\ge 5$ be an integer. Then $K_{p+3}$ contains a spanning subgraph $G'_{8}$ which is a subdivision of $G_{8}$ obtained by subdividing the edge $\overline{78}$ by $p-5$ vertices. Since $G_{8}$ is a minor of $K_{p+3}$, by Theorem \ref{maing8910} and Lemma \ref{minor_minlk}, $\Lambda(K_{p+3}) = \Psi_{G_{8},K_{p+3}}^{(2)}(\Lambda(G_{8}))$ is minimally linked. Since any element in $\Lambda(G_{8})$ is a Hamiltonian $(5,3)$-pair of cycles where the edge $\overline{78}$ is included in the $5$-cycle component, $\Lambda(K_{p+3})$ is a subset of $\Gamma_{p,3}(K_{p+3})$ consisting of exactly twelve elements.

(2) Let $p\ge 3$ be an integer. In the case of $p=3$, $K_{7}$ contains a spanning subgraph $P'_{7}$ that is equal to $P_{7}$. Since $P_{7}$ is a minor of $K_{7}$, by Lemma \ref{min_petersen} and Lemma \ref{minor_minlk}, $\Lambda(K_{7}) = \Psi_{P_{7},K_{7}}^{(2)}(\Lambda(P_{7}))$ is minimally linked. In the case of $p=4$, $K_{8}$ contains a spanning subgraph $Q'_{8}$ that is equal to $Q_{8}$. Since $Q_{8}$ is a minor of $K_{8}$, by Lemma \ref{min_petersen} and Lemma \ref{minor_minlk}, $\Lambda(K_{8}) = \Psi_{Q_{8},K_{8}}^{(2)}(\Lambda(Q_{8}))$ is minimally linked. Note that $\sharp \Lambda(K_{7}) = \sharp \Gamma_{3,4}(P_{7}) = 9$ and $\sharp \Lambda(K_{8}) = \sharp \Gamma_{4,4}(Q_{8}) = 9$. In the case of  $p\ge 5$, $K_{p+4}$ contains a spanning subgraph $G'_{9}$ which is a subdivision of $G_{9}$ obtained by subdividing the edge $\overline{78}$ by $p-5$ vertices. Then in the same way as (1), we can obtain a minimally linked subset $\Lambda(K_{p+4})= \Psi_{G_{9},K_{p+4}}^{(2)}(\Lambda(G_{9}))$ of $\Gamma_{p,4}(K_{p+4})$ consisting of exactly nine elements. 

(3) Let $p,q\ge 5$ be two integers. Then $K_{p+q}$ contains a spanning subgraph $G'_{10}$ which is a subdivision of $G_{10}$ obtained by subdividing the edge $\overline{78}$ by $p-5$ vertices and $\overline{9\ 10}$ by $q-5$ vertices. Since $G_{10}$ is a minor of $K_{p+q}$, by Theorem \ref{maing8910} and Lemma \ref{minor_minlk}, $\Lambda(K_{p+q}) = \Psi_{G_{10},K_{p+q}}^{(2)}(\Lambda(G_{10}))$ is minimally linked. Since any element in $\Lambda(G_{10})$ is a Hamiltonian $(5,5)$-pair of cycles containing both $\overline{78}$ and $\overline{9\ 10}$ in each of the components separately, $\Lambda(K_{p+q})$ is a subset of $\Gamma_{p,q}(K_{p+q})$ consisting of exactly eighteen elements. 
\end{proof}

\begin{Remark}\label{peter}
Note that $K_{10}$ contains a spanning subgraph $P'_{10}$ that is equal to $P_{10}$. Since $P_{10}$ is a minor of $K_{10}$, by Lemma \ref{min_petersen} and Lemma \ref{minor_minlk}, $\Lambda'(K_{10}) = \Psi_{P_{10},K_{10}}^{(2)}(\Lambda(P_{10}))$ is minimally linked. Since $\sharp \Lambda'(K_{10}) = 6$, the minimally linked subset $\Lambda(K_{10})$ of $\Gamma_{5,5}(K_{10})$ found in Theorem \ref{main} (3) does not realize the minimum number of elements in a minimally linked set of Hamiltonian $(5,5)$-pairs of cycles. It is still open to determine the minimum and maximum numbers of elements in a minimally linked set of Hamiltonian $(p,q)$-pairs of cycles for $K_{n}$. 
\end{Remark}

\begin{Remark}\label{knotted}
For a graph $G$, let us denote the set of all cycles of $G$ by $\Gamma(G)$, and the set of all $p$-cycles of $G$ by $\Gamma_{p}(G)$. A subset $\Gamma$ of $\Gamma(G)$ is said to be {\it knotted} if for any element $f$ in ${\rm SE}(G)$ there exists an element $\gamma$ of $\Gamma$ such that the knot $f(\gamma)$ is nontrivial, and a knotted set $\Gamma$ is said to be {\it minimally knotted} if every proper subset of $\Gamma$ is not knotted. Note that $G$ is said to be {\it intrinsically knotted} if there exists a knotted subset $\Gamma$ of $\Gamma(G)$. It is known that $\Gamma_{7}(K_{7})$ is minimally knotted \cite{CG83}, and for $n\ge 8$, $\Gamma_{n}(K_{n})$ is also knotted \cite{BBFHL07}, but an example of a minimally knotted subset of $\Gamma_{n}(K_{n})$ has not been known yet \cite[Problem 5.3]{ildt12}. See \cite[\S 5]{ildt12} for related open problems. 
\end{Remark}

\section*{Acknowledgment}

The author is grateful to Professors Ayumu Inoue and Kouki Taniyama for their valuable comments.

{\normalsize
}


\begin{thebibliography}{99}

\bibitem{AMT13}
L. Abrams, B. Mellor and L. Trott, 
Counting links and knots in complete graphs, 
{\it Tokyo J. Math.} {\bf 36} (2013), 429--458.




\bibitem{BBFHL07}
P. Blain, G. Bowlin, J. Foisy, J. Hendricks and J. LaCombe, 
Knotted Hamiltonian cycles in spatial embeddings of complete graphs, 
{\it New York J. Math.} {\bf 13} (2007), 11--16. 



\bibitem{CG83}
J. H. Conway and C. McA. Gordon, 
Knots and links in spatial graphs, 
{\it J. Graph Theory} {\bf 7} (1983), 445--453. 









\bibitem{FM09}
T. Fleming and B. Mellor, 
Counting links in complete graphs, 
{\it Osaka J. Math.} {\bf 46} (2009), 173--201. 






\bibitem{MN19_b}
H. Morishita and R. Nikkuni, 
Generalization of the Conway-Gordon theorem and intrinsic linking on complete graphs, {\it Ann. Comb.} {\bf 25} (2021), 439--470. 





\bibitem{ildt12}
T. Otsuki (Ed.), Problems on Low-dimensional Topology, 2012. 
\begin{verbatim}
https://www.kurims.kyoto-u.ac.jp/~ildt/2012/prob2012.pdf
\end{verbatim}




\bibitem{RST95}
N. Robertson, P. Seymour and R. Thomas, 
Sachs' linkless embedding conjecture,  
{\it J. Combin. Theory Ser. B} {\bf 64} (1995), 185--227. 


\bibitem{S84}
H. Sachs, 
On spatial representations of finite graphs, 
{\it Finite and infinite sets, Vol. I, II (Eger, 1981),} 649--662, 
Colloq. Math. Soc. Janos Bolyai, {\bf 37}, {\it North-Holland, Amsterdam,} 1984.


\bibitem{YL10}
A. Yu. Vesnin and  A. V. Litvintseva, 
On linking of hamiltonian pairs of cycles in spatial graphs (in Russian), 
{\it Sib. \`{E}lektron. Mat. Izv.} {\bf 7} (2010), 383--393




\end{thebibliography}
\end{document}